\documentclass[a4paper]{amsart}

\usepackage[notcite,notref]{showkeys}

\usepackage{amsthm}
\usepackage{amsmath,amsfonts,amssymb}
\usepackage{graphics,color}
\usepackage{mathrsfs}
\usepackage[bookmarksnumbered, bookmarksopen, colorlinks, citecolor=blue, linkcolor=blue]{hyperref}
\usepackage{color}

\theoremstyle{plain}
\newtheorem{theorem}{Theorem}[section]
\newtheorem{lemma}[theorem]{Lemma}

\newtheorem{corollary}[theorem]{Corollary}

\theoremstyle{remark}
\newtheorem{definition}[theorem]{Definition\rm}
\newtheorem{remark}[theorem]{Remark}

\numberwithin{equation}{section}

\newcommand*{\R}{\ensuremath{\mathbb{R}}}

\newcommand*{\N}{\ensuremath{\mathbb{N}}}

\newcommand*{\T}{\ensuremath{\mathscr{T}}}

\newcommand*{\Q}{\ensuremath{\mathcal{Q}}}
\renewcommand*{\O}{\ensuremath{\mathcal{O}}}
\renewcommand*{\L}{\ensuremath{\mathcal{L}}}
\newcommand*{\I}{\ensuremath{\mathcal{I}}}
\newcommand*{\D}{\ensuremath{\mathcal{D}}}

\begin{document}

\title[SBV-like regularity for general systems]{SBV-like regularity for  general hyperbolic systems of conservation laws}

\author{Stefano Bianchini}
\address{SISSA, via Bonomea 265, 34136 Trieste, ITALY}
\email{bianchin@sissa.it}
\urladdr{http://people.sissa.it/{\raise.17ex\hbox{$\scriptstyle\sim$}}bianchin}

\author{Lei Yu}
\address{SISSA, via Bonomea 265, 34136 Trieste, ITALY}
\email{yulei@sissa.it}

\thanks{The authors thank Laura Caravenna for his kindly help and her suggestions}

\begin{abstract}
We prove the SBV regularity of the characteristic speed of the scalar hyperbolic conservation law and SBV-like regularity of the eigenvalue functions of the Jacobian matrix of flux function for general systems of conservation laws.

More precisely,  for the equation
\begin{equation*} 
u_t + f(u)_x = 0, \quad u : \R^+ \times \R \to \Omega \subset \R^N, 
\end{equation*}
we only assume the flux $f$ is $C^2$ function in the scalar case ($N=1$) and Jacobian matrix $Df$ has distinct real eigenvalues in the system case $(N\geq 2)$. Using the modification of the main decay estimate in \cite{Lau} and localization method applied in \cite{R}, we show that for the scalar equation $f'(u)$ belongs to SBV, and for system of conservation laws the scalar measure
\[
\big( D_u \lambda_i(u) \cdot r_i(u) \big) \big( l_i(u) \cdot u_x \big)
\]
has no Cantor part, where $\lambda_i$, $r_i$, $l_i$ are the $i$-th eigenvalue, $i$-th right eigenvector and $i$-th left eigenvector of the matrix $Df$.
\end{abstract}

\maketitle

\section{Introduction}

The study of the regularity of solutions to a general system of hyperbolic system of conservation laws
\begin{equation}
\label{e:basic}
u_t + f(u)_x = 0, \quad u : \R^+\times \R \to \Omega \subset \R^N
\end{equation}
with initial data
\begin{equation}
\label{e:initial}
u(t=0) = u_0 \in \mathrm{BV}(\R,\Omega)
\end{equation}
is an important topic in the study of hyperbolic equations. In particular, recently there have been interesting advances in the analysis of the structure of the measure derivative $D_x u(t)$ of a BV solution to genuinely nonlinear scalar equations and hyperbolic systems. The results obtained are that, in addition to the BV bounds, the solution enjoys the strong regularity property that no Cantor part in the space derivative of $u(t)$ appears out of a countable set of times \cite{AD,Lau,R}: the fact that the measure $D_x u(t)$ has only absolutely continuous and jump part yields by definition that $u(t) \in \mathrm{SBV}$.

The main idea of the proof is to find a bounded functional, which is monotonically decreasing in time: then one shows that at each time a Cantor part appears the functional has a jump downward, and hence one concludes that the SBV regularity of $u$ outside a countable set of times.

This paper concerns the extension of the results of \cite{Lau} to the case where the system is only strictly hyperbolic, i.e. no assumption on the nonlinear structure of the eigenvalues $\lambda_i$ of $Df$ is done. Clearly, by just considering a linearly degenerate eigenvalue, it is fairly easy to see that the solution $u$ itself cannot be in SBV, so the regularity concerns some nonlinear function of $u$.

We state the main theorem of this paper: in the following a BV function on $\R$ will be considered defined everywhere by taking the right continuous derivative.

\begin{theorem}\label{t2}
Let $u$ be a vanishing viscosity solution of the Cauchy problem for the strictly hyperbolic system \eqref{e:basic3} with small {\rm BV} norm. Then there exists an at most countable set $S \subset \R^+$ such that the measure
\[
\big( D_u \lambda_i(u) \cdot r_i(u) \big) \big( l_i(u) \cdot u_x \big)
\]
has no Cantor part for every $t \in \R^+ \setminus S$ and $i \in \{1,\ 2,\ \dots,\ N\}$.
\end{theorem}

In the scalar case the above theorem can be rewritten as

\begin{theorem}
\label{t:scalar}
Suppose that $u \in \mathrm{BV}(\R^+\times\R)$ is an entropy solution of the scalar conservation law \eqref{e:basic2}. Then there exists a countable set $S \subset \R^+$ such that for every $t \in \R^+ \setminus S$ the following holds:
\begin{equation*}
f'(u(t))\in \mathrm{SBV}_{\mathrm{loc}}(\R).
\end{equation*}
\end{theorem} 

Since in the genuinely nonlinear case $u \mapsto \lambda_i(u)$ is invertible along the $i$-th admissible curves $T^i_s[u]$ (see Theorem \ref{t:ec} for the definition), it follows that Theorem \ref{t:me} is an extension of the results contained in \cite{Lau} (and Theorem \ref{t:scalar} is an extension of the results contained in \cite{R} when the source is $0$). The example contained in Remark \ref{ex:nonsbv} shows that the results are sharp.

The main point of the paper is the fact that the wave-tracking approximation for the waves of a genuinely nonlinear family does not essentially differ from the wavefront approximations of genuinely nonlinear systems: in other words, the wave pattern of a genuinely nonlinear characteristic family for a (approximate) solution in a general hyperbolic system has the same structure as if all characteristic families are genuinely nonlinear. Thus the analysis carried out in \cite{Lau} holds also in this case.

The proof of the above two theorems is done as follows. To introduce the argument in the easiest setting, in Section \ref{s:scalar}, we give a proof for the SBV regularity of the characteristic speed for the general scalar conservation laws. The proof is just a slight modification of the proof of Theorem 1.1 in \cite{R}.

As one sees in the proof of Theorem \ref{t:scalar}, the main tool is to obtain the SBV regularity when only one characteristic field is genuinely nonlinear (Corollary \ref{c:sri}). By inspection, the analysis of \cite{Lau} relies on the wave-front tracking approximation of \cite{Bre}, which assumes that all characteristic fields are genuinely nonlinear or linearly degenerate. Thus we devote Sections \ref{s:gpn}, Section \ref{Ss_wavefront} to introduce the wave-front tracking approximation for general systems \cite{AM}.

The focus of Section \ref{Ss_decay} is the observation that the convergence and regularity estimates of Theorem 10.4 of \cite{Bre} still holds for the $i$-th component of $u_x$, under the only assumption that the $i$-th characteristic field is genuinely nonlinear: these estimates are needed in order to define the $i$-th $(\epsilon_1,\epsilon_0)$-shocks and to pass to the limit the estimates concerning the interaction, cancellation and jump measures. The latter is responsible for the functional controlling the SBV regularity, Theorem \ref{t:me}.

After these estimates, for completeness we repeat the proof of the decay of negative waves in Section \ref{ss_negdec}. Finally we show how to adapt the strategy of the scalar case in Section \ref{s:sii}.

\section{The scalar case}
\label{s:scalar}

In this section, we restrict our attention to the scalar conservation laws and motivate our general strategy with this comparatively simpler situation. Let us consider the entropy solution to the hyperbolic conservation law in one space dimension
\begin{equation}\label{e:basic2}
\begin{cases}
u_t+f(u)_x=0 & u:\R^+\times\R\to\Omega\subset\R,\ f\in C^2(\Omega,\R), \crcr
u_{|t=0}=u_0 & u_0\in \mathrm{BV}(\R,\Omega).
\end{cases}
\end{equation}

It is easy to generalize the SBV regularity result from the convex flux case to the concave case in the following sense.

\begin{lemma}\cite{R}
\label{l:g}
Suppose $f\in C^2(\R)$ and $|f''(u)|>0$. Let $u\in L^\infty(\R)$ be an entropy solution of the scalar conservation law \eqref{e:basic2}. Then exists a countable set $S \subset \R$ such that for every $\tau\in\R^+\setminus S$ the following holds:
\begin{equation*}
u(\tau,\cdot)\in \mathrm{SBV}_{\mathrm{loc}}(\R).
\end{equation*}
\end{lemma}

Further, by Volpert's Chain Rule (Theorem 3.99 of \cite{AFP}), it follows that $f'(u(\tau,\cdot)) \in \mathrm{SBV}_{\mathrm{loc}}(\R)$ for $\tau\in\R^+\setminus S$: actually, since $f'' \not= 0$, the two conditions $f'(u(\tau)) \in \mathrm{SBV}_{\mathrm{loc}}$ and $u(\tau) \in \mathrm{SBV}_{\mathrm{loc}}$ are equivalent.

Following the same argument together with the analysis in \cite{R}, we can get a SBV regularity of the slope of  characteristics for the scalar conservation law with general flux.

\begin{proof}[Proof of Theorem \ref{t:scalar}]
Recall that if $u \in \mathrm{BV}(\R^+ \times \R)$ is an entropy solution, then by the theory of entropy solutions it follows that $u_\tau(\cdot):=u(\tau,\cdot)\in \mathrm{BV}(\R)$ for all $\tau \in \R^+$.

Define the sets
\begin{align*}
J_\tau: =&~ \big\{ x \in \R:\ u(\tau,x-)\neq u(\tau,x+) \big\}, \\
F_\tau: =&~ \big\{ x \in \R:\ f''(u(\tau,x))=0 \big\},\\
C: =&~ \big\{ (\tau,\xi)\in\R^+\times\R:\ \xi \in J_\tau \cup F_\tau \big\}.
\end{align*}
Set also $C_\tau := J_\tau \cup F_\tau$ as the $\tau$-section of $C$.

Since the Cantor part $D^\mathrm{c} u_\tau$ of $D u_\tau$ and the jump part $D^\mathrm{ac} u_\tau$ of $D u_\tau$ are mutually singular, then $|D^\mathrm{c} u_\tau|(J_\tau)=0$. 
Using the fact that $f''(u_\tau)=0$ on $F_\tau$, by Volpert's Chain Rule one obtains
\begin{align*}
|D^\mathrm{c} f(u_\tau)|(C_\tau) 
\leq&~ |D^\mathrm{c} f(u_\tau)|(J_\tau)+|D^\mathrm{c} f(u_\tau)|(F_\tau) \crcr
=&~ |f''(u_\tau)D^\mathrm{c} u_\tau|(J_\tau)+|f''(u_\tau)D^\mathrm{c} u_\tau|(F_\tau) = 0.
\end{align*}

Let  $(t_0,x_0) \in \R^+ \times \R \setminus C$. Using the finite speed of propagation and the maximum principle for entropy solutions and the fact that $u_{t_0}$ is continuous at $x_0$ by the definition of $C$, it is possible to find a triangle of the form
\begin{equation}
\label{e:triangle_T}
T(t_0,x_0):=\Big{\{}(t,x):\ |x-x_0|<b_0-\bar{\lambda}(t-t_0),\ 0<t-t_0<b_0/ \bar{\lambda}\Big{\}}
\end{equation} 
such that  $f''(u(t,x))\geq c_0>0$, for any $(t,x)\in T(t_0,x_0)$. Here $c_0$ depends on $(t_0,x_0)$ and $\bar{\lambda}$ is the maximal speed of propagation, which depends only on the $L^\infty$-bound of $u_{t_0}$ (and hence only depends on the $L^\infty$-bound of $u$ by maximal principle).

In particular, in $T(t_0,x_0)$ the solution $u$  of \eqref{e:basic2} coincides with the solution of the following problem 
\begin{equation*} 
\left\{
\begin{array}{l}
w_t + f(w)_x = 0,  \\
w(t_0,x) = \begin{cases}
u_{t_0}(x) & |x-x_0|<b_0, \\
\frac{1}{2b_0}\int^{x_0+b_0}_{x_0-b_0}u_{t_0}(y)dy & |x-x_0 |\geq b_0.
\end{cases}
\end{array}
\right.
\end{equation*}
By Lemma \ref{l:g},  $w(t,\cdot)$ is SBV regular for any $t>t_0$ out of a countable set of times $S(t_0,x_0)$. Write $T_\tau(t_0,x_0):=T(t_0,x_0)\cap\{t=\tau\}$, thus $u_\tau\llcorner_{T_\tau(t_0,x_0)}$ and $f'(u_\tau)\llcorner_{T_\tau(t_0,x_0)}$ are SBV for $\tau\in ]t_0,\ t_0+b/\bar{\lambda}[\setminus S{(t_0,x_0)}$. 
  
Let $B$ be the set of all points of $\R^+\times\R\setminus C$ which are contained in at least one of these triangles. (Notice that $T(t_0,x_0)$ is a open set and does not contain the point $(t_0,x_0)$.) Let $C':=\R^+\times \R\setminus(B\cup C)$. We claim that the set $S_{C'}:=\{\tau\in\R^+:\{t=\tau\}\cap C'\neq\emptyset\}$ is at most countable. Indeed, it is enough to prove that the set $S_{K}:=\{\tau\in\R^+:\{t=\tau\}\cap C'\cap K \neq\emptyset\}$ is at most countable for every compact set $K\subset \R^+ \times \R$ when the triangles $T(t',x')$ have a base of fixed length for every $(t',x')\in C'$: it is fairly simple to see that in this case the set $
S_K$ is finite since $(t',x')$ can not be contained in any other $T(t'',x'')$ for $t'\ne t''$ and $(t'',x'')\in C'$.

Finally, let  $\{T(t_i,x_i)\}_{i\in \N}$ be a countable subfamily of the triangles covering $B$. From the previous observation on the function $u \llcorner_{T(t_i,x_i)}$, the set
\[
S_i:= \big\{ \tau: u_\tau\llcorner_{T_\tau(t_i,x_i)}\notin \mathrm{SBV}( T_\tau(t_i,x_i)) \big\}
\]
is at most countable. For any $\tau$ not in the countable set
\[
S_{C'} \cup \bigcup_{i \in \N} S_i,
\]
one obtains the following inequality:
\begin{equation}
\label{e:final_SBV}
|D^\mathrm{c} f'(u_\tau)(\R)| \leq |D^\mathrm{c} f'(u_\tau)| \bigg( \bigcup_{i \in \N} T_\tau(t_i,x_i) \bigg) + |D^\mathrm{c} f'(u_\tau)|(C_\tau)=0.
\end{equation}
This concludes the proof.
\end{proof}

By a standard argument in the theory of BV functions, we have the following result.

\begin{corollary}
Let $u \in L^\infty(\R^+\times\R)$ be an entropy solution of the scalar conservation law \eqref{e:basic2}.  Then $f'(u)\in \mathrm{SBV}_{\mathrm{loc}}(\R^+\times\R)$.
\end{corollary}

The difference is that now the function $f'(u)$ is considered as a function of two variable.

\begin{proof}
The starting point is that up to a countable set of times, $Df'(u(t,\cdot))$ has no Cantor part (Theorem \ref{t:scalar}). From the slicing theory of BV function (Theorem 3.107-108 of \cite{AFP}), we know that the Cantor part of the 2-dimensional measure $D_xf'(u)$ is the integral with respect of $t$ of the Cantor part of $Df'(u(t,\cdot))$. This concludes that $D_xf'(u)$ has no Cantor part, i.e. $D^\mathrm{c}_xf'(u)=0$.

By combining Volpert's Chain Rule and the conservation law \eqref{e:basic2}, one has
\begin{equation*}
D^\mathrm{c}_t u = -f'(u)D^\mathrm{c}_x u.
\end{equation*}
Using Volpert's rule once again, one obtains
\begin{equation*}
D^\mathrm{c}_t f'(u) = f''(u) D^\mathrm{c}_t u = f''(u)f'(u)D^\mathrm{c}_x u = f'(u)D^\mathrm{c}_x f'(u) = 0,
\end{equation*}
which concludes that also $D_t f(u)$ has no Cantor part.  
\end{proof}

\section{Notations and settings for general systems}
\label{s:gpn}

Throughout the rest of the paper, the symbol $\O(1)$ always denotes a quantity uniformly bounded by a constant depending only on the system \eqref{e:basic3}.

\subsection{Preliminary notation}

Consider the Cauchy problem
\begin{equation}
\label{e:basic3}
\begin{cases}
u_t+f(u)_x=0 & u:\R^+\times\R\to\Omega\subset\R^N,\ f\in C^2(\Omega,\R), \crcr
u_{|t=0}=u_0 & u_0\in \mathrm{BV}(\R,\Omega).
\end{cases}
\end{equation}

The only assumption is strict hyperbolicity in $\Omega$: the eigenvalues $\{\lambda_i(u)\}_{i=1}^N$ of the Jacobi matrix $A(u)=Df(u)$ satisfy
\begin{equation*}
\lambda_1(u)<\dots<\lambda_N(u), \qquad u \in\Omega.
\end{equation*}
Furthermore, as we only consider the solutions with small variation, it is not restrictive to assume $\Omega$ compact. Hence there exist constants $\{\check{\lambda}_j\}^N_{j=0}$, such that
\begin{equation}\label{lambda}
\check{\lambda}_{k-1}<\lambda_k(u)<\check{\lambda}_{k}, \qquad \forall u\in\Omega,\ k=1,\dots, N.
\end{equation}

Let  $\{r_i(u)\}_{i=1}^N$ and $\{l_j(u)\}_{j=1}^N$ be a base of right and left eigenvectors, depending smoothly on $u$, such that 
\begin{equation}\label{assumponri}
l_j(u) \cdot r_i(u) = \delta_{ij} \text{ and } |r_i(u)| \equiv 1, \quad i=1,\dots, N.
\end{equation}

\begin{definition}
For $i=1,\dots,N$, we say that the $i$-th characteristic field (or $i$-th family) is \emph{genuinely nonlinear} if
\begin{equation*}
\nabla\lambda_i(u)\cdot r_i(u)\neq 0 \quad \text{for all} \ u \in \Omega,
\end{equation*}
and we say that the $i$-th characteristic field (or $i$-th family) is \emph{linearly degenerate} if instead
\begin{equation*}
\nabla\lambda_i(u)\cdot r_i(u)=0 \quad \text{for all} \ u\in\Omega.
\end{equation*}
\end{definition}

In the following, if the $i$-th characteristic field is genuinely nonlinear, instead of \eqref{assumponri} we normalize $r_i(u)$ such that 
\begin{equation}
\label{e:gennon_orient}
\nabla\lambda_i(u)\cdot r_i(u)\equiv 1.
\end{equation}

In \cite{BB}, it is proved that if  the total variation of $u_0$ is sufficiently small, the solutions of the viscous parabolic approximation equations
\begin{equation*}
\begin{cases}
u_t+f(u)_x=\epsilon u_{xx}, \crcr
u(0,x)=u_0(x),
\end{cases}
\end{equation*}
are uniformly bounded, and the limit of $u^\epsilon$ as $\epsilon \rightarrow 0$ is called \emph{vanishing viscosity solution} of \eqref{e:basic3} and it is BV function. 

\subsection{Construction of solutions to Riemann problem}\label{RP}

The Riemann problem is the Cauchy problem \eqref{e:basic3} with piecewise constant  initial data of the form
\begin{equation}\label{e:r}
 u_0=\left\{\begin{array}{ll}
u^\mathrm{L} & \mbox{$x<0$,}\\
u^\mathrm{R} & \mbox{$x>0$.}
\end{array}\right.
\end{equation}
The solution to this problem is the key ingredient for building the front-tracking approximate solution: the basic step is the construction of the admissible \emph{elementary curve} of the $k$-th family for any give left state $u^\mathrm{L}$. 
 
A working definition of admissible elementary curves can be given by means of the following theorem.

\begin{theorem}\cite{B2,BB}
\label{t:ec}
For every $u\in\Omega$, there exist
\begin{enumerate}
\item N Lipschitz continuous curves $s \mapsto T^k_s[u]\in\Omega$, $k=1,\dots,N$, satisfying $\lim_{s\rightarrow 0}\frac{d}{ds}T^k_s[u]=r_k(u)$,
\item and N Lipschitz functions $(s,\tau) \mapsto \sigma^k_s[u](\tau)$, with $0 \leq \tau \leq s$ and $k=1,\dots,N$, satisfying $\tau \mapsto \sigma^k_s[u](\tau)$ increasing and $\sigma^k_0[u](0) = \lambda_k(u)$,
\end{enumerate}
with the following properties.

\noindent When $u^\mathrm{L}\in\Omega,\ u^\mathrm{R}=T_s^k[u^\mathrm{L}]$, for some $s$ sufficiently small, the unique vanishing viscosity solution of the Riemann problem \eqref{e:basic3}-\eqref{e:r} is defined a.e. by 
\begin{equation*}
u(t,x) :=
\begin{cases} 
u^\mathrm{L} & x/t < \sigma^k_s[u^\mathrm{L}](0), \crcr
T^k_\tau[u^\mathrm{L}] & x/t=\sigma^k_s[u^\mathrm{L}](\tau), \tau \in [0,s], \crcr
u^\mathrm{R} & x/t>\sigma^k_s[u^\mathrm{L}](s).
\end{cases}
\end{equation*}
\end{theorem}


\begin{remark}
If $i$-th family is genuinely nonlinear, then the Lipschitz curve $T^i_s[\bar{u}]$ can be written as
\begin{equation*}\label{e:gc}
T^i_s[\bar{u}] =
\begin{cases}
R_i[\bar{u}](s) & s \geq 0, \crcr
S_i[\bar{u}](s) & s<0,
\end{cases}
\end{equation*}
where $R_i[\bar{u}]$, $S_i[\bar{u}]$ are respectively the rarefaction curve and the Rankine-Hugoniot curve of the $i$-th family with any given point $\bar{u}$ in $\Omega$. And certain elementary weak solution, called rarefaction waves and shock waves can be defined along the rarefaction curve and Rankine-Hugoniot curve, for example see \cite{Bre}. The elementary curve $T^i_s[\bar{u}]$ is parametrized by 
\begin{equation}\label{parameter of T}
 s=l_i(\bar{u})\cdot (T^i_s[\bar{u}]-\bar{u})
\end{equation}

\end{remark}

The vanishing viscosity solution \cite{BB} of a Riemann problem for \eqref{e:basic3} is obtained by constructing a Lipschitz continuous map
\begin{equation*}
(s_1,\dots,s_N)\mapsto T^N_{s_N}\big[T^{N-1}_{s_{N-1}}\big[\cdots\left[T^1_{s_1}[u^\mathrm{L}]\right]\big]\big]=u^\mathrm{R},
\end{equation*}
which is one to one from a neighborhood of the origin onto a neighborhood of $u^\mathrm{L}$. Then we can uniquely determine intermediate states $u^\mathrm{L}=\omega_0,\omega_1,\dots,\omega_N = u^\mathrm{R}$, and the \emph{wave sizes} $s_1,s_2,\dots,s_N$ such that
\begin{equation*}
\omega_k = T^k_{s_k}[\omega_{k-1}], \quad k=1,\dots,N,
\end{equation*}
provided that $|u^\mathrm{L}-u^\mathrm{R}|$ is sufficiently small.

By Theorem \ref{t:ec}, each Riemann problem with initial date
\begin{equation}\label{e:erp}
u_0 =
\begin{cases}
\omega_{k-1} & x<0, \\
\omega_k & x>0,
\end{cases}
\end{equation}
admits a vanishing viscosity solution $u_k$, containing a sequence of rarefactions, shocks and  discontinuities of the $k$-th family: we call $u_k$ the $k$-th \emph{elementary composite wave}. Therefore, under the strict hyperbolicity assumption, the general solution of the Riemann problem with the initial data \eqref{e:r} is obtained by piecing together the vanishing viscosity solutions of the elementary Riemann problems given by \eqref{e:basic3}-\eqref{e:erp}.

Indeed, from the uniform hyperbolicity assumption \eqref{lambda}, the speed of each elementary $k$-th wave in the solution $u_k$ is inside the interval $[\check{\lambda}_{k-1},\check{\lambda}_{k}]$ if $s \ll 1$, so that the solution of the general Riemann problem \eqref{e:basic3}-\eqref{e:r} is then given by
\begin{equation}
\label{e:riemann solution}
u(t,x) =
\begin{cases}
         u^\mathrm{L} & x/t <\check{\lambda}_{0}\\
         u_k(t,x) & \check{\lambda}_{k-1}<x/t<\check{\lambda}_{k}, k=1,\dots,N),\\
         u^\mathrm{R} & x/t>\check{\lambda}_{N}.
         \end{cases}
\end{equation}

\begin{remark}
If the characteristic fields are either genuinely nonlinear or linearly degenerate, the admissible solution of Riemann problem \eqref{e:basic3}-\eqref{e:r} consists of N family of waves. Each family contains either only one shock, one rarefaction wave or one contact discontinuity. However, the general solution of a Riemann problem provided above may contain a countable number of rarefaction waves, shock waves and contact discontinuities.
\end{remark}

\subsection{\texorpdfstring{Cantor part of the derivative of characteristic for $i$-th waves}{Cantor part of the derivative of characteristic for i-th waves}}

Recalling the solution \eqref{e:riemann solution} to the Riemann problem \eqref{e:basic3}-\eqref{e:r}, we denote $\tilde{\lambda}_i(u^\mathrm{L},u^\mathrm{R})$ as the $i$-th eigenvalue of the average matrix 
\begin{equation}\label{average matrix}
A(u^\mathrm{L},u^\mathrm{R})=\int^1_0 A(\theta u^\mathrm{L}+(1-\theta)u^\mathrm{R})d\theta,
\end{equation}
and $\tilde{l}_i(u^\mathrm{L},u^\mathrm{R})$, $\tilde{r}_i(u^\mathrm{L},u^\mathrm{R})$ are the corresponding left and right eigenvector satisfying $\tilde{l}_i \cdot \tilde{r}_i = \delta_{ij}$ and $|\tilde{r}_j| \equiv 1$, for every $i,j\in\{1,\dots,N\}$. Define thus
\begin{subequations}
\label{e:tilde_all}
\begin{equation}
\label{d:l}
\tilde{\lambda}_i(t,x)=\tilde{\lambda}_i(u(t,x-),u(t,x+)),
\end{equation}
\begin{equation}
\label{e:tilde_r}
\tilde{r}_i(t,x)=\tilde{r}_i(u(t,x-),u(t,x+)),
\end{equation}
\begin{equation}
\label{e:tilde_l}
\tilde{l}_i(t,x)=\tilde{l}_i(u(t,x-),u(t,x+)).
\end{equation}
\end{subequations}

Since the $\tilde r_i$, $\tilde l_i$ have directions close to $r_i$, $l_i$, one can decompose $D_xu$ into the sum of N measures:
\begin{equation*}
D_x u = \sum_{k=1}^N v_k \tilde{r}_k.
\end{equation*} 
where $v_i = \tilde{l}_i \cdot D_xu$ is a scalar valued measure which we call as \emph{$i$-th wave measure} \cite{Bre}.

In the same way we can decompose the a.c. part $D^{\mathrm{ac}}_x u$, the Cantor part $D^{\mathrm{c}}_x u$ and the jump part $D^{\mathrm{jump}}_x u$ of $D_x u$ as
\[
D^{\mathrm{ac}}_x u = \sum_{k=1}^N v_k^{ac} \tilde{r}_k, \quad
D^{\mathrm{c}}_x u = \sum_{k=1}^N v_k^{\mathrm{c}} \tilde{r}_k, \quad
D^{\mathrm{jump}}_x u = \sum_{k=1}^N v_{k}^{\mathrm{jump}} \tilde{r}_k.
\]
We call $v_i^\mathrm{c}$ the Cantor part of $v_i$ and denote by
\begin{equation*}
v^{\mathrm{cont}}_{i} := v^\mathrm{c}_i + v^{ac}_i=\tilde{l}_i\cdot(D^\mathrm{c}_xu+D^\mathrm{ac}_i u)
\end{equation*}
the continuous part of $v_i$. According to Volpert's Chain Rule
\begin{equation}
D_x\lambda_i(u)=\nabla\lambda_i(u)(D^\mathrm{ac}_x u + D^\mathrm{c}_x u) + [\lambda_i(u^+)-\lambda_i(u^-)] \delta_x,
\end{equation}
and then
\begin{equation}
D^\mathrm{c}_x\lambda_i(u)=\nabla\lambda_i \cdot D_x^\mathrm{c}u=\sum_k \big( \nabla\lambda_i\cdot\tilde{r}_k \big) v_k^\mathrm{c}.
\end{equation}

We define the \emph{$i$-th component of $D_x\lambda_i(u)$} as
\begin{equation}
\label{E_i_th_comp_lambda}
[D_x\lambda_i(u)]_i:= \big( \nabla\lambda_i \cdot \tilde{r}_i \big) v^\mathrm{cont}_i + [\lambda_i(u^+)-\lambda_i(u^-)] \frac{|v^\mathrm{jump}_i(x)|}{\sum_k |v^{\mathrm{jump}}_k(x)|},
\end{equation}
and the \emph{Cantor part of $i$-th component of $D_x\lambda_i(u)$} to be
\begin{equation*}
[D^\mathrm{c}_x\lambda_i(u)]_i:= \big( \nabla\lambda_i \cdot \tilde{r}_i \big) v^\mathrm{c}_{i}.
\end{equation*}

\section{Main SBV regularity argument}
\label{s:main result}

Following \cite{Lau}, the key idea to obtain SBV-like regularity for $v_i$ is to prove a decay estimate for the continuous part of $v_i$. We state here the main estimate of our paper.

\begin{theorem}\label{t:me}
Consider the general strictly hyperbolic system \eqref{e:basic3}, and suppose that the $i$-th characteristic field  is genuinely nonlinear. Then there exists a finite, non-negative Radon measure $\mu^\mathrm{ICJ}_i$ on $\R^+\times\R$ such that for $t>\tau>0$
\begin{equation}\label{e:me}
\big{|}v^\mathrm{cont}_i(t)\big{|}(B)\leq\mathcal{O}(1)\bigg{\{}\frac{\mathcal{L}(B)}{\tau}+\mu^\mathrm{ICJ}_i([t-\tau,t+\tau]\times\R)\bigg{\}}
\end{equation}
for all Borel subset $B$ of $\R$.
\end{theorem}
 
Different from \cite{Lau}, we assume only one characteristic field to be genuinely nonlinear and no other requirement on the other characteristic fields. 

Once Theorem \ref{t:me} is proved, then the SBV argument develops as follows \cite{Lau}.

Suppose at time $t=s$, $v_i(s)$ has a Cantor part. Then there exists a $\L^1$-negligible Borel set $K$ with $v^\mathrm{cont}_i(s)(K)>0$ and $D^\mathrm{jump} v_i(K)=0$. Then for all $s>\tau>0$,
\begin{equation*}
0< |v_i(s)|(K) = |v^\mathrm{cont}_i(s)|(K) \leq \O(1)\bigg{\{}\frac{\L^1(K)}{\tau}+\mu^\mathrm{ICJ}_i([s-\tau,s+\tau]\times\R)\bigg{\}}.
\end{equation*}
Since $\L^1(K)=0$, we can let $\tau\rightarrow 0$, and deduce that $\mu^\mathrm{ICJ}_i(\{s\}\times\R) > 0$. This shows that the Cantor part appears at most countably many times because $\mu^\mathrm{ICJ}_i$ is finite.

Then, we can have the following result which generalizes Corollary 3.2 in \cite{Lau} to the case when only one characteristic field is genuinely nonlinear and no assumptions on the others.

\begin{corollary}\label{c:sri}
Let u be a vanishing viscosity solution of the Cauchy problem for the strictly hyperbolic system \eqref{e:basic3}, and assume that the  $i$-th characteristic field is genuinely nonlinear. Then $v_i(t)$ has no Cantor part out of a countable set of times.
\end{corollary}

As we see in the scalar case, by proving the SBV regularity of the solution under the genuinely nonlinear assumption of one characteristic field, we can deduce a kind of SBV regularity of the characteristic speed for general systems.

Unlike the scalar case,  we do not have the maximum principle to guarantee the small variation of $u$ in the triangle $T(t_0,x_0)$ defined in \eqref{e:triangle_T}. However, in the system case, we have the following estimates for the vanishing viscosity solutions.

For $a<b$ and $\tau\geq0$, we denote by $\mathrm{Tot.Var.}\{u(\tau);\ ]a,b[\}$ the total variation of $u(\tau)$ over the open interval $]a,b[$. Moreover, consider the triangle
\begin{equation*}
\Delta^{\tau,\eta}_{a,b} := \Big\{ (t,x):\ \tau < t < (b-a)/2\eta,\ a + \eta t < x < b - \eta t \Big\}.
\end{equation*}
The oscillation of $u$ over $\Delta^{\tau,\eta}_{a,b}$ will be denoted by
\begin{equation*}
\mathrm{Osc.}\{u;\ \Delta^{\tau,\eta}_{a,b}\}:=\sup\left\{|u(t,x)-u(t',x')|:\ \ (t,x),\ (t',x')\in\Delta^{\tau,\eta}_{a,b}\right\}.
\end{equation*}

We have the following results.

\begin{theorem}[Tame Oscillation]\cite{BB}\label{t:to}
 There exists $C'>0$ and $\bar{\eta}>0$ such that for every $a<b$ and $\tau\ge0$, one has
\begin{equation*}
\mathrm{Osc.}\{u;\ \Delta^{\tau,\bar{\eta}}_{a,b}\}\leq C'\cdot \mathrm{Tot.Var.}\{u(\tau);\ ]a,b[\}.
\end{equation*}
\end{theorem}

Adapting the proof of the scalar case, we can prove the main Theorem \ref{t2} of this paper: the proof of this theorem will be done in Section \ref{s:sii}.

\section{Review of wave-front tracking approximation for general system}
\label{s:ftm}

To prove Theorem \ref{t:me}, we use the front tracking approximation in \cite{AM} which extends the one in \cite{Bre} to the general systems. Since the construction is now standard, we only give a short overview about existence, compactness and convergence of the approximation, pointing to the properties needed in our argument: more precisely, we will only consider how one construct the approximate wave pattern of the $k$-th genuinely nonlinear family (Section \ref{Ss_k_gnl}).

The main point is that, for general systems, the accurate/simplified/crude Riemann solvers for the $k$-th wave coincides with the approximate/simplified/crude Riemann solvers when all families are genuinely nonlinear (see below for the definition of accurate/simplified/crude Riemann solvers). This means that the wave pattern pf the $k$-th genuinely nonlinear family will have the same structure as if all other families are genuinely nonlinear: by this, we mean that shock-shock interaction generates shocks, the jump in characteristic speed across $k$-th waves is proportional to their size, and one can thus use the $k$-component of the derivative of $\lambda_k$ \eqref{E_i_th_comp_lambda} to measure the total variation of $v_k$.

\subsection{Description of front tracking approximation}
\label{Ss_wavefront}
Front tracking approximation is an algorithm which produces  piecewise constant approximate solutions to the Cauchy problem \eqref{e:basic3}. Roughly speaking, we first choose a piecewise constant function $u^\epsilon_0$ which is a good approximation to initial data $u_0$ such that
\begin{equation}\label{initial approx}
\mathrm{Tot.Var.}\{u^\epsilon_0\}\leq \mathrm{Tot.Var.}\{u_0\}, \quad ||u^\epsilon_0-u_0||_{L^1}<\epsilon,
\end{equation}
and $u^\epsilon_0$ only has finite jumps.  Let $x_1<\dots<x_m$ be the jump points of $u^\epsilon_0$. For each $\alpha=1,\dots,m$, we approximately solve the Riemann problem (see section \ref{RP}, just shifting the center from $(0,0)$ to $(0,x_\alpha)$) with the initial data given by the jump $[u^\epsilon_0(x_\alpha-),[u^\epsilon_0(x_\alpha+)]$ by a function $w(t,x)=\phi(\frac{x-x_0}{t-t_0})$ where $\phi$ is a piecewise constant function. The straight lines where the discontinuities locate are called $wave$-$fronts$ (or just \emph{fronts} for short). The wave-fronts can prolong until they interact with other fronts, then at the interaction point, the corresponding Riemann problem is approximately solved and several new fronts are generated forward. Then one tracks the wave-fronts until they interact with other wave-fronts, etc... In order to avoid the algorithm to produce infinite many wave-fronts in finite time, different kinds of approximate Riemann solvers should be introduced.

\subsubsection{Approximate Riemann solver}

There are two kinds of approximate Riemann solvers defined for interactions between two physical wave-fronts.
Suppose at the point $(t_1,x_1)$, a wave-front of size $s'$ belonging to $k'$-th family interacts from the left with a wave-front of size $s''$ belonging to $k''$-th family for some $k',\ k''\in \{1,\cdots,N\}$ such that
\[
u^M=T^{k'}_{s'}[u^\mathrm{L}],\qquad u^\mathrm{R}=T^{k''}_{s''}[u^M].
\]
 Assuming that $|u^\mathrm{L}-u^\mathrm{R}|$ sufficiently small.
At the interaction point, the Riemann problem with the initial data data $[u^\mathrm{L},u^\mathrm{R}]$ will be solved by approximate Riemann solver. 

\begin{itemize}
 \item\emph{Accurate Riemann Solver} replaces each elementary composite wave of the exact Riemann solution (refers to $u_k$ in \eqref{e:riemann solution}) by an approximate elementary wave  which is
 a finite collection of jumps traveling with a speed given by the average speed $\tilde \lambda_k$ given by\eqref{d:l}, and the wave opening (i.e. the difference in speeds between any two consecutive fronts) is less than some small parameter $\epsilon$ controlling the
 accuracy of the approximation. 
 
\vspace{6pt}
 
 \item\emph{Simplified Riemann Solver} only generates approximate elementary waves belong to $k'$-th and $k''$-th families with corresponding size $s'$ and $s''$ as the incoming ones if $k'\ne k''$, and approximate elementary waves of size $s'+s''$ belong to $k'$-th
 family if $k'= k''$. The simplified Riemann solver collects the remaining new waves into a single \emph{nonphysical front}, traveling with a constant speed $\hat{\lambda}$, strictly larger than all characteristic speed $\hat{\lambda}$. Therefore, usually the simplified Riemann solver generate less outgoing fronds after interaction than the accurate Riemann solver.
\end{itemize}

Since the simplified Riemann solver produces nonphysical wave-fronts and they can not interact with each other, one only needs a approximate Riemann solver defined for the  interaction between, for example, a physical front of the $k$-th family with size $s$, connecting $u^M$, $u^\mathrm{R}$ and a nonphysical front (coming from the left)  connecting the left value $u^\mathrm{L}$ and $u^M$ traveling with speed $\hat{\lambda}$.
\begin{itemize}
\item\emph{Crude Riemann Solver} generates a $k$-th front connecting $u^\mathrm{L}$ and $\tilde{u}^M=T^k_s[u^\mathrm{L}]$ traveling with speed $\tilde{\lambda_i}$ and a nonphysical wave-front joining $\tilde{u}^M$ and $u^\mathrm{R}$, traveling with speed $\hat{\lambda}$.  In the following, for simplicity, we just say that the non-physical fronts belong to the $(N+1)$-th characteristic field. 
\end{itemize}

\begin{remark}
We can assume that at each time $t>0$, at most one interaction takes place, involving exactly two incoming fronts, because we can slightly change the speed of one of the incoming fronts if more than two fronts meet at the same point. It is sufficient to require that the error vanishes when $\epsilon \to 0$. \\
To simplify the analysis, we assume that the fronts satisfy the Rankine-Hugoniot conditions exactly.
\end{remark}

\subsubsection{The approximate Riemann solvers for genuinely nonlinear waves}
\label{Ss_k_gnl}

If the $k$-th characteristic family is genuinely nonlinear, the elementary wave $u_k$ is either a shock wave or a rarefaction wave. The key example of the accurate Riemann solver is thus to consider how these two solutions are approximated.

If $k$-th elementary wave $u_k$ in \eqref{e:riemann solution} is just a single shock, for example
\[
 u_k=\begin{cases}
      u^\mathrm{L}\quad &x/t<\sigma,\\
      u^\mathrm{R}\quad &x/t>\sigma,
     \end{cases}
\]
where $\sigma$ is the speed of shock wave, then the approximated $k$-th wave coincides the exact one (apart from the speed in case, see the above remark).

If $u_k$ is a rarefaction wave of the $k$-th family connecting the left value $u^\mathrm{L}$ and the right value $u^\mathrm{R}$, for example, if $u^\mathrm{R}:=T_{s}^k[u^\mathrm{L}]$ and
\[
  u_k=\begin{cases}
      u^\mathrm{L}\quad &x/t<\lambda_i(u^\mathrm{L}),\\
      T_{s^*}^k[u^\mathrm{L}] &x/t\in[\lambda_i(u^\mathrm{L}),\lambda_i(u^\mathrm{R})],\ x/t=\lambda_i(T_{s^*}^k[u^\mathrm{L}]),\\
      u^\mathrm{R}\quad &x/t>\lambda_i(u^\mathrm{R}),
     \end{cases}
\]
where $s^*\in[0,s]$. Then the approximation $\tilde{u}_k$ is a rarefaction fan containing several rarefaction fronts. More precisely, we can choose real numbers $0=s_0<s_1<\dots<s_n=s$, and define the points $w_i:=T_{s_i}^k[u^\mathrm{L}]$, $i=0,\dots,n$, with the following properties,
\begin{eqnarray*}
w_{i+1}=T_{(s_{i+1}-s_i)}^k[w_{i}],
\end{eqnarray*}
and the wave opening of consecutive wave-fronts are sufficiently small, i.e.
\begin{equation*}
\sigma^k_{s}[u^\mathrm{L}](s_{i+1})-\sigma^k_{s}[u^\mathrm{L}](s_i)\leq\epsilon, \quad \forall i=0,\dots,n-1.
\end{equation*}
where the function $\sigma^k_{s}$ is defined in Theorem \ref{t:ec}. We let the jump $[\omega_i,\omega_{i+1}]$ travel with the speed $\tilde{\sigma}_i:=\tilde{\lambda}_k(\omega_i,\omega_{i+1})$ \eqref{d:l}, so that the rarefaction fan $\tilde{u}_k$ becomes 
\begin{equation*}
\tilde{u}_k=\begin{cases}
u^\mathrm{L} & x/t<\tilde{\sigma}_1,\\
\omega_i & \tilde{\sigma}_{i} \leq x/t<\tilde{\sigma}_{i+1}, \ i=1,\dots,n-1,\\
u^\mathrm{R} & x/t\geq\tilde{\sigma}_n.
\end{cases}
\end{equation*}

\subsubsection{Interaction potential and BV estimates}
\label{interaction amount}

Suppose two wave-fronts with size $s'$ and $s''$ interact. In order to get the estimate on the difference between the size of the incoming waves and the size of the outgoing waves produced by the interaction, we need to define the \emph{amount of interaction $\mathcal{I}(s',s'')$} between $s'$ and $s''$.

When $s'$ and $s''$ belong to different characteristic families, set
\[
\mathcal{I}(s',s'')=|s's''|.
\]

If $s'$, $s''$ belong to the same characteristic family, the definition of $\mathcal I(s',s'')$ is more complicated (see Definition 3 in \cite{AM}). We just mention that if $s'$, $s''$ are the sizes of two shocks which have the same sign, traveling with the speed $\sigma'$ and $\sigma''$ respectively, then the Amount of Interaction takes the form
\begin{equation}\label{d:ai}
\mathcal{I}(s',s'')=|s's''|\big{|}\sigma'-\sigma''\big{|},
\end{equation}
i.e. the product of the size of the waves times the difference of their speeds (of the order of the angle between the two shocks).

To control the Amount of Interaction, the following potential is introduced.

At each time $t>0$ when no interaction occurs, and $u(t,\cdot)$ has jumps at $x_1,\dots,x_m$, we denote by
\[
\omega_1,\dots,\omega_m, \quad s_1,\dots,s_m, \quad i_1,\dots,i_m,
\]
their left states, signed sizes and characteristic families, respectively: the sign of $s_\alpha$ is given by the respective orientation of $dT^k_s[u]/ds$ and $r_k$, if the jump at $x_\alpha$ belongs to the $k$-th family. The Total Variation of $u$ will be computed as
\[
V(t) = V(u(t)) := \sum_{\alpha} \big| s_\alpha \big|.
\]

Following \cite{B2}, we define the \emph{Glimm Wave Interaction Potential} as follows:
\begin{equation}\label{d:gp}
\begin{split}
\Q(t) &= \mathcal{Q}(u(t)) := \sum_{\genfrac{}{}{0pt}{}{i_\alpha>i_\beta}{x_\alpha<x_\beta}} \big| s_\alpha s_\beta \big| + \frac{1}{4} \sum_{i_\alpha=i_\beta<N+1} \int^{|s_\alpha|}_0\int^{|s_\beta|}_0 \big| \sigma^{i_\beta}_{s_\beta}[\omega_\beta](\tau'')-\sigma^{i_\alpha}_{s_\alpha}[\omega_\alpha](\tau') \big| d\tau'd\tau''.
\end{split}
\end{equation}

Denoting the time jumps of the Total Variation and the Glimm Potential as
\[
\Delta V(\tau)=V(\tau+)-V(\tau-),\ \ \Delta \Q(\tau)=\Q(\tau+)-\Q(\tau-),
\] 
the fundamental estimates are the following (Lemma 5 in \cite{AM}): in fact, when two wave-fronts with size $s',\ s''$ interact,
\begin{subequations}
\label{e:gpe_ve}
\begin{equation}
\label{e:gpe}
\Delta\Q(\tau)=-\mathcal{O}(1)\mathcal{I}(s',s''),
\end{equation}
\begin{equation}
\label{e:ve}
\Delta V(\tau)=\mathcal{O}(1)\mathcal{I}(s',s'').
\end{equation}
\end{subequations}
Thus one defines the \emph{Glimm Functional}
\begin{equation}
\label{e:glimm_funct}
\Upsilon(t) := V(t) + C_0 \Q(t)
\end{equation}
with $C_0$ suitable constant, so that $\Upsilon$ decreases at any interaction. Using this functional, one can prove that $\epsilon$-approximate solutions exist and their total variations are uniformly bounded (see section 6.1 of \cite{AM}).

\subsubsection{Construction of the approximate solutions and their convergence to exact solution}

The construction starts at initial time $t=0$ with a given $\epsilon>0$, by taking  $u^\epsilon_0$ as a suitable piecewise constant approximation of initial data $u_0$, satisfying \eqref{initial approx}. At the jump points of $u^\epsilon_0$, we locally solve the Riemann problem by accurate Riemann solver. The approximate solution $u^\epsilon$ then can be prolonged until a first time $t_1$ when two wave-fronts interact. Again we solve the Riemann problem at the interaction point by an approximate Riemann solver. Whenever the amount of interaction (see Section \ref{interaction amount} for the definition) of the incoming waves is larger than some threshold parameter $\rho = \rho(\epsilon) > 0$, we shall adopt the accurate Riemann solver. Instead, in the case where the amount of interaction of the incoming waves is less than $\rho$, we shall adopt two different types of simplified Riemann solvers. And we will apply the crude Riemann solver if one of the incoming wave-front is non-physical front. One can show that the number of wave-fronts in approximate solution constructed by such algorithm remains finite for all times (see Section 6.2 in \cite{AM}).

We call such approximate solutions \emph{$\epsilon$-approximate front tracking solutions}. At each time $t$ when there is no interaction, the restriction $u^\epsilon(t)$ is a step function whose jumps are located along straight lines in the $(t,x)$-plane. 

Let $\{\epsilon_\nu\}^\infty_{\nu=1}$ be a sequence of positive real numbers converging to zero. Consider a corresponding sequence of $\epsilon_\nu$-approximate front tracking solutions $u^\nu:=u^{\epsilon_\nu}$ of \eqref{e:basic3}: it is standard to show that the functions $t\mapsto u^\nu(t,\cdot)$ are uniformly Lipschitz continuous in $L^1$ norm, and the decay of the Glimm Functional yields that the solutions $u^\nu(t,\cdot)$ have uniformly bounded total variation. Then by Helly's theorem, $u^\nu$ converges up to a subsequence in $\mathbb{L}^1_{\mathrm{loc}}(R^+\times\R)$ to some function $u$, which is a weak solution of \eqref{e:basic3}.

It can be shown that by the choice of the Riemann Solver in Theorem \ref{t:ec}, the solution obtained by the front tracking approximation coincides with the unique vanishing viscosity solution \cite{BB}. Furthermore, there exists a closed domain $\D\subset L^1(\R,\Omega)$ and a unique distributional solution $u$, which is a Lipschitz semigroup $\D\times[0,+\infty[\rightarrow \D$ and which for piecewise constant initial data coincides, for a small time, with the solution of the Cauchy problem obtained piecing together the standard entropy solutions of the Riemann problems. Moreover, it lives in the space of BV functions.
  
For simplicity, the pointwise value of $u$ is its $L^1$ representative such that the restriction map $t\mapsto u(t)$ is continuous form the right in $L^1$ and $x \mapsto u(t,x)$ is right continuous from the right.

\subsubsection{Further estimates}

%
To each $u^\nu$, we define the \emph{measure $\mu^\mathrm{I}_\nu$ of interaction} and the \emph{measure $\mu^\mathrm{IC}_\nu$ of interaction and cancellation} concentrated on the set of interaction points as follows. If two physical front fronts belonging to the families $i,i'\in\{1,\dots,N\}$ with size $s',\ s''$ interact at point $P$, we denote
\begin{equation*}
\mu^\mathrm{I}_\nu(\{P\}):=\mathcal{I}(s',s''),
\end{equation*}

\begin{equation*}
\mu^\mathrm{IC}_\nu(\{P\}) \;:=\mathcal{I}(s',s'')+\;
\left\{\begin{array}{ll}
|s'|+|s''|-|s'+s''| & \mbox{$i=i'$},\\
0 & \mbox{$i\neq i'$}.
\end{array}\right.
\end{equation*}

The wave size estimates (Lemma 1 in \cite{AM}) yields balance principles for the wave size of approximate solution. \\
More precisely, given a polygonal region $\Gamma$ with edges transversal to the waves it encounters. Denote by $W^{i\pm}_{\nu,\mathrm{in}}$, $W^{i\pm}_{\nu,\mathrm{out}}$ the positive $(+)$ or negative $(-)$ $i$-th waves in $u^\nu$ entering or exiting $\Gamma$, and let $W^i_{\nu,\mathrm{in}}=W^{i+}_{\nu,\mathrm{in}}-W^{i-}_{\nu,\mathrm{in}}$, $W^i_{\nu,\mathrm{out}}=W^{i+}_{\nu,\mathrm{out}}-W^{i-}_{\nu,\mathrm{out}}$. Then the measure of interaction and the measure of interaction-cancellation control the difference between the amount of exiting $i$-th waves and the amount of entering $i$-th waves w.r.t. the region as follows:
\begin{subequations}
\label{e:bl_bl_1}
\begin{equation}
\label{e:bl}
|W^i_{\nu,\mathrm{out}}-W^i_{\nu,\mathrm{in}}|\leq\mathcal{O}(1)\mu^\mathrm{I}_\nu(\Gamma),
\end{equation}
\begin{equation}
\label{e:bl_1}
|W^{i\pm}_{\nu,\mathrm{out}}-W^{i\pm}_{\nu,\mathrm{in}}|\leq\mathcal{O}(1)\mu^\mathrm{IC}_\nu(\Gamma).
\end{equation}
\end{subequations}
The above estimates are fairly easy consequence of the interaction estimates \eqref{e:gpe_ve} and the definition of $\mu^\mathrm{I}_\nu$, $\mu^\mathrm{IC}_\nu$.

By taking a subsequence and using the weak compactness of bounded measures, there exit measures $\mu^{I}$ and $\mu^\mathrm{IC}$ on $\R^+\times\R$ such that the following weak convergence holds:
\begin{equation}\label{def of muic}
\mu^\mathrm{I}_{\nu}\rightharpoonup\mu^\mathrm{I}, \quad \mu^\mathrm{IC}_\nu\rightharpoonup\mu^\mathrm{IC}.
\end{equation}

\subsection{\texorpdfstring{Jump part of $i$-th waves}{Jump part of i-th waves}}
\label{Ss_decay}

The derivative of $u^\nu$ is clearly concentrated on polygonal lines, being a piecewise constant function with discontinuities along lines. To select the fronts of $u^\nu$ converging to the jump part of $u$, we use the following definition.

\begin{definition}[Maximal $(\epsilon^0,\epsilon^1)$-shock front]\cite{Bre}
A maximal $(\epsilon^0,\epsilon^1)$-shock front for the $i$-th family of an $\epsilon_\nu$-approximate 
front-tracking solution $u^\nu$ is any maximal (w.r.t. inclusion) polygonal line
$(t,\gamma^\nu(t))$ in the $(t,x)$-plane, $t_0\leq t \leq t_1$, satisfying:
\begin{itemize}
\item[(i)] the segments of $\gamma^\nu$ are i-shocks of $u^\nu$ with size $|s^\nu|\geq \epsilon^0$, and at least once $|s^\nu|\geq \epsilon^1$;
\item[(ii)] the  nodes are interaction points of $u^\nu$;
\item[(iii)] it is on the left of any other polygonal line which it intersects and which have the above two properties.
\end{itemize}
\end{definition}

Let $M^{\nu,i}_{(\epsilon^0,\epsilon^1)}$ be the number of maximal $(\epsilon^0,\epsilon^1)$-shock front for the $i$-th family. Denote 
\begin{equation*}
\gamma^{\nu,i}_{(\epsilon^0,\epsilon^1),m}:[t^{\nu,i,-}_{(\epsilon^0,\epsilon^1),m},t^{\nu,i,+}_{(\epsilon^0,\epsilon^1),m}]\rightarrow\R, \quad m=1,\dots,M^{\nu,i}_{(\epsilon^0,\epsilon^1)},
\end{equation*}
as the maximal $(\epsilon^0,\epsilon^1)$-shock fronts for the $i$-th family in $u^\nu$. Up to a subsequence, we can assume that $M^{\nu,i}_{(\epsilon^0,\epsilon^1)} = \bar{M}^i_{(\epsilon^0,\epsilon^1)}$ is a constant independent of $\nu$ because the total variations of $u^\nu$ are bounded.

Consider the collection of all maximal $(\epsilon^0,\epsilon^1)$-shocks for the $i$-th family and define
\begin{equation*}
\mathscr{T}^{\nu,i}_{(\epsilon^0,\epsilon^1)}=\bigcup^{\bar M^i_{(\epsilon^0,\epsilon^1)}}_{m=1} \mathrm{Graph}(\gamma^{\nu,i}_{(\epsilon^0,\epsilon^1),m}),
\end{equation*}
and let $\{\epsilon_k^0\}_{k\in\N}$, $\{\epsilon^1_k\}_{k\in\N}$ be two sequences satisfying $0 < 2^k \epsilon^0_k \leq \epsilon^1_k \searrow 0$.

Up to a diagonal argument and by a suitable labeling of the curves, one can assume that for each fixed $k$, $m$ the Lipschitz curves $\gamma^{\nu,i}_{(\epsilon^0_k,\epsilon^1_k),m}$ converge uniformly to a Lipschitz curve $\gamma^i_{(\epsilon^0_k,\epsilon^1_k),m}$. 
Let
\begin{equation*}
\mathscr{T}^i := \bigcup_{m,k}\mathrm{Graph}(\gamma^i_{(\epsilon^0_k,\epsilon^1_k),m}).
\end{equation*}
denote the collection of all these limiting curves in $u$.

For fixed $(\epsilon^0,\epsilon^1)$, we write for shortness
\begin{equation}\label{d:lnu}
\tilde{l}^\nu_i(t,x):=\tilde{l}_i(u^\nu(t,x-),u^\nu(t,x+))
\end{equation}
and define
\begin{equation}\label{d:vnujump}
v^{\nu,\mathrm{jump}}_{i,(\epsilon^0,\epsilon^1)}:=\tilde{l}^\nu_i\cdot {u^\nu_x\llcorner}_{\T^{\nu,i}_{(\epsilon^0,\epsilon^1)}}.
\end{equation}

Following the same idea of the proof of Theorem 10.4 in \cite{Bre}, the next lemma holds if only the $i$-th characteristic field is genuinely nonlinear.

\begin{lemma}
\label{l:is}
The jump part of $v_i$ is concentrated on $\T^i$.

Moreover there exists a countable set $\Theta \subset \R^+ \times \R$, such that for each point
\[
P=(\tau,\xi)=(\tau,\gamma^i_m(\tau))\notin\Theta
\]
where $i$-th shock curve $\gamma^i_m$ is approximated by the sequence of $(\epsilon^0,\epsilon^1)$-shock fronts $\gamma^{\nu,i}_{(\epsilon^0,\epsilon^1),m}$ of the approximate solutions $u^\nu$, the following holds
\begin{subequations}
\label{e:left_right_lim}
\begin{equation}
\label{e:leftlim}
\lim_{r \rightarrow 0+} \limsup_{\nu \rightarrow \infty} \left( \sup_{\genfrac{}{}{0pt}{}{x<\gamma^{\nu,i}_{(\epsilon^0,\epsilon^1),m}(t)}{(t,x)\in B(P,r)}} \big| u^\nu(t,x) - u^- \big| \right) = 0,
\end{equation}
\begin{equation}
\label{e:rightlim}
\lim_{r \rightarrow 0+} \limsup_{\nu \rightarrow \infty} \left( \sup_{\genfrac{}{}{0pt}{}{x>\gamma^{\nu,i}_{(\epsilon^0,\epsilon^1),m}(t)} {(t,x)\in B(P,r)}} \big| u^\nu(t,x) - u^+ \big| \right) = 0.
\end{equation}
\end{subequations}
Moreover, we can choose a sequence $\{\nu_k\}_{k=1}^\infty$ such that
\begin{equation}\label{converge of vijump}
v^\mathrm{jump}_i = \mathrm{weak}^*\mathrm{-}\lim_k\sum^N_{i=1}v^{\nu_k,\mathrm{jump}}_{i,(\epsilon^0_k,\epsilon^1_k)}.
\end{equation}
\end{lemma}

The key argument of the proof is that we can use the tools of the proof of Theorem 10.4 in \cite{Bre} because the wave structure of the $i$-th genuinely nonlinear family has the following properties:
\begin{enumerate}
\item the interaction among two shocks of the $i$-th family generates only one shock of the $k$-th family,
\item the strength of $i$-th waves can be measured by the jump of the $i$-th characteristic speed $\lambda_i$,
\item the speed of $i$-th waves is very close to the average of the jump of $\lambda_i$ across the discontinuity.
\end{enumerate}
These properties are a direct consequence of the behavior of the approximate Riemann solvers on the $i$-th waves, if the $i$-th family is genuinely nonlinear (Section \ref{Ss_k_gnl}).

\begin{proof}[Sketch of the proof]
Let $\Theta$ be the set defined by all jump points of the initial datum, the atoms of $\mu^\mathrm{IC}$ (see \eqref{def of muic}). 

For any point $P\in \T^i\setminus\Theta$, if \eqref{e:leftlim} or \eqref{e:rightlim} does not hold, then this means that the approximate solutions $u^\nu$ has some uniform oscillation: indeed, because of the $L^1$-convergence in $\R^2$, one can find points $(t,x)$ arbitrarily close to $P$ such that $u^\nu(t,x) \to u(t,x)$, and the fact that the limits are not $0$ means that there are other points $(t',x')$ arbitrarily close to $P$ such that $|u^\nu(t',x') - u(t,x)|$, at least for a subsequence of $\nu$. The analysis of the proof of Theorem 10.4 in \cite{Bre} shows that in this case the Amount of Interaction and Cancellation is uniformly positive in every neighborhood of $P$, and thus $P$ is an atom of $\mu^\mathrm{IC}$, contradicting to $P\notin \Theta$.

For $P\notin \T^i \cup \Theta$, if $v^\mathrm{jump}_i(P)>0$, i.e. $P$ is a jump point of $u$, by the similar argument of Step 8 in the proof of Theorem 10.4 in \cite{Bre} this shows that the waves present in the approximate solutions are canceled, and thus $\mu^\mathrm{IC}(P)>0$. It is impossible since $P\notin\Theta$. This concludes that $v^\mathrm{jump}_i$ is concentrated on $\T^i$, because by \eqref{e:left_right_lim} the jumps in the approximate solutions are vanishing in a neighborhood of every $P \notin \T^i \cup \Theta$.

We are left with the proof of \eqref{converge of vijump}. At jump point $(t,\gamma^i_{(\epsilon^0,\epsilon^1),m}(t)) \in \T^i\setminus\Theta$, according to \eqref{e:leftlim}, \eqref{e:rightlim}, there exist a sequence $(t^\nu,\gamma^{\nu,i}_{(\epsilon^0,\epsilon^1),m}(t)(t^\nu))$ such that 
\begin{equation}
\label{convgoflefteigen}
\big( t,\gamma^{i}_{(\epsilon^0,\epsilon^1),m}(t) \big) = \lim_{\nu \to \infty} \big( t^\nu,\gamma^{\nu,i}_{(\epsilon^0,\epsilon^1),m}(t)(t^\nu) \big)
\end{equation}
and its left and right values converges to the left and right values of the jump in $(t,\gamma^i_{(\epsilon^0,\epsilon^1),m}(t))$.

Since $f\in C^2$, by the definition \eqref{average matrix} the matrix $A(u^\mathrm{L},u^\mathrm{R})$ depends continuously on the value $(u^\mathrm{L},u^\mathrm{R})$, and since its eigenvalues are uniformly separated the same continuity holds for its eigenvalues $\tilde \lambda_k(u^\mathrm{L},u^\mathrm{R})$, left eigenvectors $\tilde l_k(u^\mathrm{L},u^\mathrm{R})$ and right eigenvectors $\tilde r_k(u^\mathrm{L},u^\mathrm{R})$. Using the notation \eqref{d:l} and \eqref{d:lnu}, one obtains
\begin{equation}
\label{convg of left eigen}
\tilde{l}_i \big( t,\gamma^{i}_{(\epsilon^0,\epsilon^1),m}(t) \big)=\lim_\nu \tilde{l}^\nu_i \big( t^\nu,\gamma^{\nu,i}_{(\epsilon^0,\epsilon^1),m}(t^\nu) \big),
\end{equation}
and similar limits holds for $\tilde r_i$, $\tilde \lambda_i$.

Up to a subsequence $\{\nu_k\}$, from the convergence of the graphs of $\mathscr{T}^{\nu_k,i}_{(\epsilon_k^0,\epsilon_k^1)}$ to $\T^{i}$ and \eqref{e:leftlim}, \eqref{e:rightlim}, it is fairly easy to prove that
\begin{equation}
\label{convg of uxjump}
D u \llcorner_{\T^{i}} = \lim_{k \to \infty} D u^\nu \llcorner_{\T^{\nu_k,i}_{(\epsilon_k^0,\epsilon_k^1)}}.
\end{equation}
According to  \eqref{d:vnujump}, \eqref{convg of left eigen} and \eqref{convg of uxjump}, one concludes the weak convergence of $v^{\nu_k,\mathrm{jump}}_{i,(\epsilon_k^0,\epsilon_k^1)}$ to $v^\mathrm{jump}_i$. 
\end{proof}

\section{Proof of Theorem \ref{t:me}}
\label{s:me}

\subsection{Decay estimate for positive waves}

The Glimm Functional for BV functions to general systems has been obtained in \cite{B2}, and when $u$ is piecewise constant, it reduced to \eqref{d:gp}: and we will write it as $\Q$ also the formulation of the functional given in \cite{B2}. Moreover, for the same constant $C_0>0$ of the Glimm Functional $\Upsilon(t)$ \eqref{e:glimm_funct}, the sum $\mathrm{Tot.Var.}(u)+C_0\Q(u)$ is lower semi-continuous w.r.t the $L^1$ norm (see Theorem 10.1 of \cite{Bre}).

For any Radon measure $\mu$, we denote $[\mu]^+$ and $[\mu]^-$ as the positive and negative part of $\mu$ according to Hahn-Jordan decomposition. The same proof of the decay of the Glimm Functional $\Upsilon(t)$ yields that for every finite union of the open intervals $J=I_1\cup\dots\cup I_m$
\begin{equation}
\label{e:lsc_glimm}
[v_i]^\pm(J)+C_0\Q(u) \leq \liminf_{\nu\rightarrow\infty}\left\{[v^\nu_i]^\pm(J)+C_0\Q(u^\nu)\right\}, \quad  i=1,\dots,\ n,
\end{equation}
as $u^\nu\rightarrow u$ in $L^1$. 

In \cite{Bre,MR1632980} the authors prove a decay estimate for positive part of the $i$-th wave measure under the assumption that $i$-th characteristic field is genuinely nonlinear and the other characteristic fields are either genuinely nonlinear or linearly degenerate. By inspection, one can verify that the proof also works (with a little modification) under no assumptions on the nonlinearity on the other characteristic fields, since the essential requirements of strict hyperbolicity and of the controllability of interaction amounts by Glimm Potential still hold: the main variation is that one should replace the original Glimm Potential in \cite{Bre} with the generalized one given in \cite{B2}.

We thus state the following theorem, which is the analog of Theorem 10.3 in \cite{Bre}.

\begin{theorem}
\label{t:bde}
Let the system \eqref{e:basic} be strictly hyperbolic and the $i$-th characteristic field be genuinely non-linear. Then there exists a constant $C''$ such that, for every $0\leq s<t$ and every solution $u$ with small total variation obtained as the limit of wave-front tracking approximation, the measure $[v_i(t)]^+$ satisfies
\begin{equation}
\label{vibde}
[v_i(t)]^+(B)\leq C''\left\{\frac{\mathcal{L}^1}{t-s}(B)+[\Q(s)-\Q(t)]\right\}
\end{equation}
for every $B$ Borel set in $\R$.
\end{theorem}

The estimate \eqref{vibde} given half of the bound \eqref{e:me}.

\subsection{Decay estimate for negative waves}
\label{ss_negdec}

To simplify the notation, we omit the index $(\epsilon^0,\epsilon^1)$ in $v^{\nu,\mathrm{jump}}_{i,(\epsilon^0,\epsilon^1)}$ in the rest of the proof. In order to get the uniform estimate for the \emph{continuous part} $v^{\nu,\mathrm{cont}}_{i} := v^\nu_i - v^{\nu,\mathrm{jump}}_{i}$, we need to consider the distributions
\begin{equation*}
\mu^\nu_i:=\partial_t v^\nu_i+\partial_x(\tilde{\lambda}^\nu_i v^\nu_i) ,\ \ \ \ \mu^{\nu,\mathrm{jump}}_{i} := \partial_t v^{\nu,\mathrm{jump}}_{i} + \partial_x(\tilde{\lambda}^{\nu}_i  v^{\nu,\mathrm{jump}}_{i}).
\end{equation*}

\subsubsection{\texorpdfstring{Estimate for $\mu^\nu_i$}{Estimate for the source}}

%
Let $y_m : [\tau^-_m,\tau^+_m] \rightarrow \R$, $m = 1,\dots,L^\nu$, be time-parameterized segments whose graphs are the $i$-th wave-fronts of $u^\nu$ and define
\begin{equation*}
u_m^\mathrm{L}:=u(t,y_m(t)-), \quad u_m^\mathrm{R}=u(t,y_m(t)+), \quad \ t \in ]\tau^-_m,\tau^+_m[. 
\end{equation*}
For any test function $\phi \in C^\infty_c (\R^+ \times \R)$ on obtains 
\begin{equation}\label{e:mu_i_nu_est}
- \int_{\R^+ \times \R} \phi \, d\mu^\nu_i = \sum_{m=1}^{L^\nu} \big[ \phi(\tau^+_m,y_m(\tau^+_m))-\phi(\tau^-_m, y_m(\tau^-_m) \big] \tilde{l}_i \cdot (u_m^\mathrm{R} - u_m^\mathrm{L}).
\end{equation}

For any $m$, since the $i$-th characteristic field is genuinely nonlinear, one has
\[
 |\tilde{l}_i(u^\mathrm{L},u^\mathrm{R})-l_i(u^\mathrm{L})|=\O(1)|u_m^\mathrm{R}-u_m^\mathrm{L}|,
\]
where $u_m^\mathrm{R}=T^i_{s_i}[u_m^\mathrm{L}]$ for some size $s_i$. Then it follows from \eqref{parameter of T} that 
\begin{equation}\label{equiv of size and strength}
 s_i\cong \tilde{l}_i \cdot (u_m^\mathrm{R} - u_m^\mathrm{L}). 
\end{equation}

%
%
%
%
%

Let $\{(t_k,x_k)\}_k$ be the collection of points where the $i$-th fronts interact. 
The computation \eqref{e:mu_i_nu_est} yields that $\mu^\nu_i$ concentrates on the interaction points, i.e.
\begin{equation*}
\mu^\nu_i=\sum_k p_k \delta_{(t_k,x_k)},
\end{equation*}
where $p_k$ is the difference between the strength of the $i$-th waves leaving at $(t_k,x_k)$ and the $i$-th waves arriving at $(t_k,x_k)$. We estimate the quantity $p_k$ depending on the type of interaction:

Since  in \cite{Lau}, it is proved that the total size of nonphysical wave-fronts are of the same order of $\epsilon_\nu$, when decomposing $u^\nu_x$, we only consider the physical fronts. If at $(t_k,x_k)$, two physical fronts with $i$-th component size $s'_i$, $s''_i$ interact and generate an $i$-th wave or a rarefaction fan with total size $s_i = \sum_m s_i^m$, from \eqref{e:mu_i_nu_est} and \eqref{equiv of size and strength}, one has
\begin{equation}\label{interaction of pw}
p_k\cong s_i-s'_i-s''.
\end{equation}
Notice that $s'$ or $s''$ or both may vanish in \eqref{interaction of pw} if one of incoming physical fronts does not belong to the $i$-th family.

%

According to the estimate in \cite{AM} (Lemma 1), the difference of sizes between the incoming and outgoing waves of the same family is controlled by the Amount of Interaction (see Section \ref{interaction amount}), so that one concludes
\begin{equation*}
|\mu^\nu_i|(\{(t_k,x_k)\}) \leq \O(1) \I(s_i,s'_i)
\end{equation*}
and thus
\[
|\mu^\nu_i|(\{t_k\} \times \R) \leq \O(1) \{\Upsilon^\nu(t_k^-)-\Upsilon^\nu(t_k^+)\}.
\]
This yields
\begin{equation}
|\mu^\nu_i|(\R^+\times\R) \leq \O(1) \Upsilon^\nu(0),
\end{equation}
i.e. $|\mu^\nu_i|$ is a finite Radon measure.

\subsubsection{\texorpdfstring{Estimate for $\mu^{\nu,\mathrm{jump}}_{i}$}{Estimate for the jump part}}

Let $\gamma^i_m : [\tau^-_m,\tau^+_m] \to \R$, $m=1,\dots,\bar M^i_{(\epsilon^0,\epsilon^1)}$, be the curves whose graphs are the segments supporting the fronts of $u^\nu$ belonging to $\T^{\nu,i}_{(\epsilon^0,\epsilon^1)}$, and write 
\[
u_m^\mathrm{L} := u \big( t,\gamma^i_m(t)- \big), \quad u_m^\mathrm{R} := u \big( t,\gamma^i_m(t)+ \big), \qquad t \in ]\tau^-_m,\tau^+_m[.
\]
For any test function $\phi\in C^\infty_c(\R^+ \times \R)$ by direct computation one has as in \eqref{e:mu_i_nu_est}
\begin{equation}
-\int_{\R^+ \times \R} \phi\, d\mu^{\nu,\mathrm{jump}}_{i} = \sum_{m=1}^{\bar M^i_{(\epsilon^0,\epsilon^1)}} \big[ \phi(\tau^+_m,y_m(\tau^+_m)) - \phi(\tau^-_m,y_m(\tau^-_m) \big] \tilde{l}_i \cdot (u^\mathrm{R}-u^\mathrm{L}),
\end{equation}
which yields 
\begin{equation*}
\mu^{\nu,\mathrm{jump}}_{i}=\sum_k q_k\delta_{(\tau_k,x_k)},
\end{equation*} 
where $(\tau_k,x_k)$ are the nodes of the jumps in $\T^{\nu,i}_{(\epsilon^0,\epsilon^1)}$ and the quantities $q_k$ can be computed as follows: if the $i$-th incoming waves have sizes $s'$ and $s''$, and the outgoing $i$-th shock has size $s$, then (see \cite{Lau})
\begin{equation}
\label{e:q_k}
q_k \cong
\begin{cases}
-s' & (t_k,x_k) \ \text{terminal point of a front not merging into another front}, \\
s & (t_k,x_k) \ \text{initial point of a maximal front}, \\
s-s'-s'' & (t_k,x_k) \ \text{merging point of two fronts}, \\
s-s' & (t_k,x_k) \ \text{interaction point of a front with waves not belonging to} \ \T^{\nu,i}_{(\epsilon^0,\epsilon^1)}.
\end{cases} 
\end{equation}
Except in the case when $(\tau_k,x_k)$ is the terminal points of the front of  $\mathscr{T}^{\nu,i}_{(\epsilon^0,\epsilon^1)}$ which ends without merging into another, one has by the interaction estimates
\begin{equation*}
q_k\leq \mu_\nu^\mathrm{IC}(\tau_k,x_k).
\end{equation*}
In fact, since $s \leq 0$ on shocks the second case of \eqref{e:q_k} implies $q_k \leq 0$.

Suppose now that $(\tau_k,x_k)$ is a terminal point of an $(\epsilon^0,\epsilon^1)$-shock front $\gamma_m$. By the definition of $(\epsilon^0,\epsilon^1)$-shock, for some $t \leq \tau_k$ the shock front $\gamma_m$ has size $s_0 \leq -\epsilon^1$, and at $(\tau_k,x_k)$ the size $s_1$ of the outgoing $i$-th front must be not less than $-\epsilon^0$ as a result of interaction-cancellation among waves. Hence we obtain
\begin{equation*}
\epsilon^1 - \epsilon^0 \leq |s_0| - |s_1| \leq \O(1) \mu^\mathrm{IC}_\nu(\gamma_k).
\end{equation*}
This yields
\begin{align*}
q_k \cong -s_1+(s_1+q_k) \leq&~ \frac{\epsilon^0}{\epsilon^1-\epsilon^0} (\epsilon^1-\epsilon^0) + \O(1) \mu^\mathrm{I}_\nu(t_k,x_k) 
\leq \frac{\O(1) \epsilon^0}{\epsilon^1-\epsilon^0}\mu^\mathrm{IC}_\nu(\gamma_k)+\O(1)\mu^\mathrm{I}_\nu(t_k,x_k).
\end{align*}
Since the end points correspond to disjoint maximal $i$-th fronts, due to genuinely nonlinearity, it follows that
%
\begin{equation*}
\sum_{(t_k,x_k)\ \text{end point}} q_k \leq \O(1) \mu^\mathrm{IC}_\nu (\R^+ \times \R),
\end{equation*}
so that it is a uniformly bounded measure. We thus conclude that the distribution
\begin{equation*}
\bar{\mu}^\nu :=-\mu^{\nu,\mathrm{jump}}_{i} + \O(1)\mu^\mathrm{IC}_\nu + \sum_{(t_k,x_k)\ \text{end point}} q_k \delta_{(t_k,x_k)}
\end{equation*}
is non-negative, so it is a Radon measure and thus also $\mu^{\nu,\mathrm{jump}}_{i}$ is a Radon measure. 

In order to obtain a lower bound, one considers the Lipschitz continuous test function
\[
\phi_\alpha(t) := \chi_{[0,T+\alpha]}(t) - \frac{t-T}{\alpha} \chi_{[T,T+\alpha]}(t), \quad \alpha > 0,
\]
which is allowed because $v_i^\nu$ is a bounded measure. Since $\bar{\mu}$ is non-negative, one obtains
\begin{align*}
\bar{\mu}^\nu \big( [0,T] \times \R \big) \leq&~ \int_{\R^+ \times \R} \phi_\alpha d\bar{\mu} \crcr
=&~ - \int_{\R^+ \times \R} \phi_\alpha d\mu^{\nu,\mathrm{jump}}_{i} + \O(1) \int_{\R^+ \times \R} \phi_\alpha d\mu^\mathrm{IC}_\nu + \sum_{(t_k,x_k)\ \text{end point}} q_k \phi_\alpha(t_k) \\
\leq&~ \int_{\R^+ \times \R} \big[ (\phi_\alpha)_t + \tilde{\lambda}^\nu_i (\phi_\alpha)_x \big] d \big[ v^{\nu,\mathrm{jump}}_{i}(t) \big] dt + \big[ v^{\nu,\mathrm{jump}}_{i}(0) \big](\R) + \O(1) \mu^\mathrm{IC}_\nu \big( [0,T+\alpha] \times \R \big) \\
\leq&~ - \frac{1}{\alpha} \int^{T+\alpha}_T \big[ v^{\nu,\mathrm{jump}}_{i}(t) \big](\R) dt + \big[ v^{\nu,\mathrm{jump}}_{i}(0) \big](\R) + \O(1) \mu^\mathrm{IC}_\nu \big( [0,T+\alpha] \times \R \big).
\end{align*}
Letting $\alpha \searrow 0$ and since $[v^{\nu,\mathrm{jump}}_{i}(\R)](0)$ is negative, one concludes
\begin{equation*}
\bar{\mu}^\nu \big([0,T] \times \R \big) \leq - \big[ v^{\nu,\mathrm{jump}}_{i}(T) \big](\R) + \O(1) \mu^\mathrm{IC}_\nu \big( [0,T+\alpha] \times \R \big) \leq \O(1) \Upsilon^\nu(0).
\end{equation*}
We conclude this section by writing the uniform estimate
\begin{equation*}
- \O(1) \Upsilon^\nu(0) \leq \mu^{\nu,\mathrm{jump}}_{i} \leq \O(1) \mu^\mathrm{IC}_\nu.
\end{equation*}
In particular, the definitions of the measures $\mu^\nu_i$, $\mu^{\nu,\mathrm{jump}}_i$ give the following balances for the $i$-th waves across the horizontal lines:
\begin{subequations}
\label{e:balance_balance_jump}
\begin{equation}
\label{balance}
\big[ v^\nu_i(t+) \big](\R) - \big[ v^\nu_i(t-) \big](\R) = \mu^\nu_i \big( \{t\} \times \R \big),
\end{equation}
\begin{equation}
\label{balance jump}
\big[ v^{\nu,\mathrm{jump}}_i(t+) \big](\R) - \big[ v^{\nu,\mathrm{jump}}_i(t-) \big](\R) = \mu^{\nu,\mathrm{jump}}_i \big( \{t\} \times \R \big).
\end{equation}
\end{subequations}
The limits are taken in the weak topology. Notice that we can always take that $t \mapsto v^nu_i(t), v^{\nu,\mathrm{jump}}_i(t)$ is right continuous in the weak topology.

\subsubsection{\texorpdfstring{Balances of $i$-th waves in the region bounded by generalized characteristics}{Balances of waves in the region bounded by generalized characteristics}}

We recall that a \emph{minimal generalized $i$-th characteristic} is an absolutely continuous curve starting from $(t_0,x_0)$ satisfying the differential inclusion
\begin{equation*}
x^\nu(t;t_0,x_0) := \min \Big\{ x^\nu(t) : x^\nu(t_0) = x_0, \ \dot{x}^\nu(t) \in \big[ \lambda_i \big( u^\nu(t,x(t)+ \big), \lambda_i \big( u^\nu(t,x(t)-) \big) \big] \Big\}
\end{equation*}
for a.e. $t \geq t_0$. Given an interval $I=[a,b]$, we define the region $A^{\nu,(t_0,\tau)}_{[a,b]}$ bounded by the minimal $i$-th characteristics $a(t)$, $b(t)$ of $u^\nu$ starting at $(t_0,a)$ and $(t_0,b)$ by
\begin{equation*}
A^{\nu,(t_0,\tau)}_{[a,b]} := \Big\{ (t,x) : t_0< t \leq t_0 + \tau, \ a(t) \leq x \leq b(t) \Big\},
\end{equation*}
and its time-section by $I(t):=[a(t),b(t)]$. Let $J:=I_1\cup I_2 \cup \dots \cup I_M$ be the union of the disjoint closed intervals $\{I_i\}_{i=1}^M$, and set
\[
J(t) := I_1(t) \cup \dots \cup I_M(t), \quad A^{\nu,(t_0,\tau)}_J := \bigcup^M_{m=1} A^{\nu,(t_0,\tau)}_{I_m}.
\]
We will now obtain wave balances in regions of the form $A^{\nu,(t_0,\tau)}_J$. Due to the genuinely non-linearity of the $i$-th family, the corresponding proof in \cite{Lau} works, we will repeat it for completeness.

The balance on the region $A^{\nu,(t_0,\tau)}_J$ has to take into account also the contribution of the flux $\Phi^\nu_i$ across boundaries of the segments $I_m(t)$: due to the definition of generalized characteristic and the wave-front approximation, it follows that $\Phi^\nu_i$ is an atomic measure on the characteristics forming the border of $A^{\nu,(t_0,\tau)}_J$, and moreover a positive wave may enter the domain $A^{\nu,(t_0,\tau)}_J$ only if an interaction occurs at the boundary point $(\hat t, \hat x)$, which gives the estimate
\begin{equation}\label{flux control}
\Phi^{\nu}_i \big( \{ (\hat{t},\hat{x}) \} \big) \leq \O(1) \mu^\mathrm{IC}_i \big( \{ (\hat{t},\hat{x}) \} \big).
\end{equation}
One thus obtains that 
\begin{equation}
\label{e:vibalance}
\big[ v^\nu_i(\tau) \big](J(\tau)) - \big[ v^\nu_i(t_0) \big](J) = \mu^\nu_i \big( A^{\nu,(t_0,\tau)}_{J} \big) + \Phi^\nu_i \big( A^{\nu,(t_0,\tau)}_{J} \big) + \O(1) \epsilon_\nu,
\end{equation}
where the last term depends on the errors due to the wave-front approximation (a single rarefaction front may exit the interval $I_m$ at $t_0$).

The same computation can be done for the jump part $v^{\nu,\mathrm{jump}}_i$, obtaining
%
\begin{equation}
\label{e:vijumpbalance}
\big[ v^{\nu,\mathrm{jump}}_{i}(J(t)) \big](\tau) - \big[ v^{\nu,\mathrm{jump}}_{i}(t_0) \big](J) = \mu^{\nu,\mathrm{jump}}_{i} \big( A^{\nu,(t_0,\tau)}_{J} \big) + \Phi^{\nu,\mathrm{jump}}_{i} \big( A^{\nu,(t_0,\tau)}_{J} \big).
\end{equation}
Since the flux $\Phi^{\nu,\mathrm{jump}}_{i}$ only involves the contribution of $(\epsilon^0,\epsilon^1)$-shocks, it is clearly non-positive.
 
Subtracting \eqref{e:vijumpbalance} to \eqref{e:vibalance}, one finds the following equation for $v^{\nu,\mathrm{cont}}_{i}$:
\begin{equation*}
\big[ v^{\nu,\mathrm{cont}}_{i}(\tau) \big](J(\tau)) - \big[ v^{\nu,\mathrm{cont}}_{i}(t_0) \big](J) = \big( \mu^\nu_i - \mu^{\nu,\mathrm{jump}}_{i} \big) \big( A^{\nu,\tau}_J \big) + \big( \Phi^{\nu}_{i} - \Phi^{\nu,\mathrm{jump}}_{i} \big) \big( A^{\nu,(t_0,\tau)}_J \big) + \O(1) \epsilon_\nu.
\end{equation*}

Denote the difference between the two fluxes by
\[
\Phi^{\nu,\mathrm{cont}}_{i} := \Phi^{\nu}_{i} - \Phi^{\nu,\mathrm{jump}}_{i}.
\]
Since $\Phi^{\nu,\mathrm{jump}}_{i}$ removes only some terms in the negative part of $\Phi^{\nu}_{i}$, one concludes that
\begin{equation}
\label{estimate for phicont}
\Phi^{\nu}_{i} - \Phi^{\nu,\mathrm{jump}}_{i} \leq \big[ \Phi^{\nu}_{i} \big]^+ \leq \mu^\mathrm{IC}_\nu.
\end{equation}
Setting
\[
\mu^\mathrm{ICJ}_{i,\nu} := \mu^\mathrm{IC}_\nu + \big| \mu^{\nu,\mathrm{jump}}_{i} \big|,
\]
and using the estimate $|\mu^\nu_i| \leq \O(1) \mu^\mathrm{IC}_\nu$, one has
\begin{equation}
\label{estimate for muicont}
\mu_i^\nu - \mu_{i}^{\nu,\mathrm{jump}} \leq \O(1) \mu^\mathrm{ICJ}_{i,\nu}.
\end{equation}

\subsubsection{Decay estimate}

Due to the semigroup property of solutions, it is sufficient to prove the estimate for the measure $[v^{\nu,\mathrm{cont}}_{i} (t=0)]^-$. Consider thus a closed interval $I = [a,b]$ and let $z(t) := a(t)-b(t)$
where
\begin{equation*}
a(t):= x^\nu(t;0,a), \quad b(t):= x^\nu(t;0,b)
\end{equation*}
and the minimal forward characteristics stating at $t=0$ from $a$ and $b$. For $\mathcal L^1$-a.e. $t$ one has
\begin{equation*}
\dot{z}(t) = \tilde{\lambda}_i(t,b(t)) - \tilde{\lambda}_i(t,a(t)).
\end{equation*}
By introducing a piecewise Lipschitz continuous non-decreasing potential $\Phi$ to control the waves on the other families \cite{Bre}, with $\Phi(0) = 1$, one obtain
\begin{equation}\label{estimate of zdot}
\Big| \dot{z}(t) + \xi(t) - \big[ v^\nu_i(t) \big](I(t)) \Big| \leq \O(1) \epsilon_\nu + \dot{\Phi}(t) z(t),
\end{equation}
where
\begin{equation*}
\xi(t) := \big( \tilde{\lambda}_i(t,a(t)) - {\lambda}_i(t,a(t)-) \big) + \big( \tilde{\lambda}_i(t,b(t)+) - {\lambda}_i(t,b(t)) \big). 
\end{equation*}
%

We consider two cases.

\noindent \emph{Case 1.} If
\[
\dot{z}(t) - \dot{\Phi}(t) z(t) < \frac{1}{4} \big[ v^\nu_i(0) \big](I)
\]
for all $t>0$, then
\begin{equation*}
\frac{d}{dt} \Big[ e^{-\int^t_0 \dot{\Phi}(\tau) d\tau} z(t) \Big] = e^{-\int^t_0\dot{\Phi}(\tau)d\tau} \big\{ \dot{z}(t) - \dot{\Phi}(t) z(t) \big\} < \frac{e^{-\int^t_0\dot{\Phi}(\tau)d\tau}}{4} \big[ v^\nu_i(0) \big](I) \leq \O(1) \big[ v^\nu_i(0) \big](I).
\end{equation*}
Integrating the above inequality from $0$ to $\tau$ and remembering that $\Phi(0) = 1$, one has
\begin{equation*}
-\L^1(I) = z(0) \leq e^{-\int^\tau_0\dot{\Phi}(\tau) d\tau} z(\tau) - z(0) \leq \O(1) \tau \big[ v^\nu_i(I) \big] \leq \O(1) \tau \big[ v^{\nu,\mathrm{cont}}_i(0) \big](I),
\end{equation*}
which reads as
\begin{equation*}
- \big[ v^{\nu,\mathrm{cont}}_i(0) \big](I) \leq \O(1) \frac{\L^1(I)}{\tau}.
\end{equation*}

\noindent \emph{Case 2.} Assume instead that
\[
\dot{z}(t) - \dot{\Phi}(t)z(t) \geq  \frac{1}{4} \big[ v^\nu_i(0) \big](I)
\]
at some time $t > 0$. By the inequality \eqref{estimate of zdot} and the balance \eqref{e:vibalance} one obtains
\begin{equation}
\label{e:case_2}
\begin{split}
\frac{1}{4} \big[ v^\nu_i(0) \big](I) \leq&~ \dot{z}(t) - \dot{\Phi}(t)z(t) \crcr
\leq&~ \big[ v^\nu_i(t) \big](I(t)) - \xi(t) + \O(1) \epsilon_\nu \crcr
\leq&~ \big[ v^\nu_i(0) \big](I) + \mu^\nu_i \big( A^{\nu,(0,t)}_{[a,b]} \big) + \Phi^{\nu}_{i} \big( A^{\nu,(0,t)}_{[a,b]} \big) - \xi(t) + \O(1)\epsilon_\nu.
\end{split}
\end{equation}
Hence
\begin{equation}\label{estimate of vi1}
\big[ v^\nu_i(0) \big](I) \geq -\frac{4}{3} \bigg[ \mu^\nu_i \big( A^{\nu,(0,t)}_{[a,b]} \big) + \Phi^{\nu}_{i} \big( A^{\nu,(0,t)}_{[a,b]} \big) - \xi(t) + \O(1)\epsilon_\nu \bigg].
\end{equation}
From \eqref{d:vnujump} and the fact that the fronts in $\T^{\nu,i}_{(\epsilon^0,\epsilon^1)}$ satisfy Rankine-Hugoniot conditions (up to a negligible error), we have
\[
v^{\nu,\mathrm{jump}}_i(t,a(t)) = \lambda_i(t,a(t)+) - \lambda_i(t,a(t)-),
\]
and since
\[
\bigg| \tilde \lambda_i(u^\mathrm L,u^\mathrm R) - \frac{\lambda_i(u^\mathrm L) + \lambda_i(u^\mathrm L)}{2} \bigg| \leq \O(1) \big| u^\mathrm L - u^\mathrm R \big|^2,
\]
by the balance \eqref{e:vijumpbalance}, we conclude that
\begin{equation}
\label{estimate of vi}
\begin{split}
\xi(t) \geq&~ \frac{3}{4} \Big[ \big[ v^{\nu,\mathrm{jump}}_{i}(t) \big]({a(t)}) + \big[ v^{\nu,\mathrm{jump}}_{i}(t) \big]({b(t)}) - 2 \epsilon^1 \Big] \\
\geq&~ \frac{3}{4} \Big[ \big[ v^{\nu,\mathrm{jump}}_{i}(t) \big](I(t)) - 2 \epsilon^1 \Big] \\
\geq&~ \frac{3}{4} \bigg[ \big[ v^{\nu,\mathrm{jump}}_{i}(0) \big](I) + \mu^{\nu,\mathrm{jump}}_{i} \big( A^{\nu,(0,t)}_{[a,b]} \big) + \Phi^{\nu,\mathrm{jump}}_{i} \big( A^{\nu,(0,t)}_{[a,b]} \big) - 2 \epsilon^1 \bigg].
\end{split}
\end{equation}

Substituting \eqref{estimate of vi} into \eqref{estimate of vi1}, using the estimates \eqref{estimate for phicont}, \eqref{estimate for muicont} and integrating \eqref{e:case_2}, we obtain
\begin{equation*}
- \big[ v^{\nu,\mathrm{cont}}_{i}(0) \big](I) \leq \O(1) \bigg\{ \frac{\mathcal{L}^1(I)}{t}+\mu^\mathrm{ICJ}_\nu \big( \overline{A^{\nu,(0,t)}_U} \big) + \epsilon^1 + \epsilon_\nu \bigg\}.
\end{equation*}
This gives the estimate \eqref{e:me} for the case of a single interval for the approximate solution.

By repeating the analysis in the case of a finite union of intervals, one obtains the same bound as above, and since $v^{\nu,\mathrm{cont}}_{i}$ is a Radon measure, the same result holds for any Borel sets, i.e.
\begin{equation*}
- \big[ v^{\nu,\mathrm{cont}}_{i}(0) \big](B) \leq \O(1) \bigg\{ \frac{\mathcal{L}^1(B)}{t} + \mu^\mathrm{ICJ}_\nu \big( \overline{A^{\nu,(0,t)}_B} \big) + \epsilon^1 + \epsilon_\nu \bigg\},
\end{equation*}
where $B$ is any Borel in $\R$ and
\[
A^{\nu,(0,t)}_B := \Big\{ \big( \tau,x^\nu(\tau;0,x_0) \big) : \ x \in B, \ 0 < \tau \leq t \Big\}.
\]

As the solution is independent on the choice of the approximation, we can consider a particular converging sequence $\{u^\nu\}_{\nu\geq 1}$ of $\epsilon_\nu$-approximate solutions with the following additional properties:
\begin{equation}
\Q(u^\nu(0,\cdot))\rightarrow\Q(u_0).
\end{equation}

By lower semi-continuity of  $[v_i(0)]^-+C_0\Q(u(0))$ \eqref{e:lsc_glimm}, one gets
\begin{equation}\label{lowsemi}
[v_i(0)]^-+C_0\Q(u(0))\leq \mathrm{weak}^*-\liminf_{\nu \to \infty}\left\{[v_i^\nu(0)]^-+C_0\Q(u^\nu(0))\right\}.
\end{equation}

From \eqref{lowsemi} and \eqref{converge of vijump}, up to a subsequence, one obtains for any open set $B \subset \R$,
\begin{equation}
\begin{split}
\big[ v^\mathrm{cont}_i(0) \big]^-(U) =&~ [v_i(0)]^-(U) + \big[ v^\mathrm{jump}_i(0) \big](U) \\
\leq&~ 
\liminf_{\nu \rightarrow \infty} \big\{ \big[ v^\nu_i(0) \big]^-(U) + C_0 \Q(u^\nu(0)) \big\} - C_0 \Q(u(0)) 
+ \lim_{\nu \rightarrow \infty} \big[ v^{\nu,\mathrm{jump}}_{i}(0) \big](U) \\
=&~ 
\liminf_{\nu \rightarrow \infty} \big\{ \big[ v^{\nu,\mathrm{cont}}_{i}(0) \big]^-(U) + C_0 \Q(u^\nu(0)) \big\} - C_0 \Q(u(0)) \\
\leq&~ \liminf_{\nu \rightarrow \infty} \mathcal{O}(1) \bigg\{ \frac{\mathcal{L}^1(U)}{t} + \mu^{\nu,\mathrm{ICJ}}_i \big( \overline{A^{\nu,(0,t)}_U} \big) + \epsilon^1 + \epsilon_\nu + \Q(u^\nu(0)) - \Q(u(0)) \bigg\}\\
\leq&~ \mathcal{O}(1) \bigg\{ \frac{\mathcal{L}^1(U)}{t} + \mu^\mathrm{ICJ}_i \big( [0,t] \times \R \big) \bigg\},
\end{split}
\end{equation}
where $\mu^\mathrm{ICJ}_i$ is defined as weak$^*$-limit of measure $\mu^{\nu,\mathrm{ICJ}}_{i}$ (up to a subsequence). Then the outer regularity of Radon measure yields the inequality for any Borel set.


The above estimate together with Theorem \ref{t:bde} gives \eqref{e:me}.

\section{\texorpdfstring{SBV regularity for the $i$-th component of the $i$-th eigenvalue}{SBV regularity for the i-th component of the i-th eigenvalue}}
\label{s:sii}

This last section concerns the proof of Theorem \ref{t2}, adapting the strategy of Section \ref{s:scalar}.

\begin{proof}[Proof of Theorem \ref{t2}]
As in the scalar case, we define the sets
\begin{align*}
J_\tau :=&~ \big\{ x \in \R : \ u^\mathrm{L}(\tau,x) \neq u^\mathrm{R}(\tau,x) \big\}, \\
F_\tau :=&~ \big\{ x \in \R : \ \nabla\lambda_i(u(\tau,x)) \cdot r_i(u(\tau,x)) = 0 \big\}, \\
C :=&~ \big\{ (\tau,\xi) \in \R^+ \times \R : \ \xi \in J_\tau\cup F_\tau \big\}, \quad C_\tau := J_\tau \cup F_\tau.
\end{align*}
By definition of continuous part
\[
\big| v^\mathrm{cont}_i(J_\tau) \big|(\tau) = 0,
\]
and since
\[
\nabla\lambda_i \big( u(\tau,F_\tau \setminus J_\tau) \big) \cdot r_i \big( u(\tau,F_\tau \setminus J_\tau) \big) = 0,
\]
we conclude that
\begin{equation*}
\begin{split}
\big| \nabla \lambda_i(u) \cdot r_i(u) v^\mathrm{cont}_i(\tau)|(C_\tau) = \big| \nabla\lambda_i(u) \cdot r_i(u) v^\mathrm{cont}_i(\tau) \big| (J_\tau) + \big| \nabla\lambda_i(u) \cdot r_i(u) v^\mathrm{cont}_i(\tau) \big|(F_\tau\setminus J_\tau) = 0.
\end{split}
\end{equation*}

For any $(t_0,x_0) \in \R^+ \times \R \setminus C$, there exist strictly positive $b_0 = b_0(x_0,t_0)$, $c_0 = c_0(x_0,t_0)$ such that
\begin{equation*}
\big| \nabla \lambda_i \cdot r_i(u(t_0,x)) \big| \geq c_0 > 0
\end{equation*}
for every $x$ in the open interval $I_0 := ]-b_0+x_0,x_0+b_0[$, because $u(t_0,x)$ is continuous at $x_0 \notin C_{t_0}$. Hence by Theorem \ref{t:to}, we know the there is a triangle
\[
T_0 := \Big\{ (t,x) :\ |x-x_0| < b'_0 - \bar{\eta}(t-t_0) , \ 0<t-t_0<b'_0/ \bar{\eta} \Big\}
\]
with the basis $I'_0 := ]-b'_0+x_0,x_0+b'_0[ \subset I_0$, such that
\begin{equation}
\label{e:uni_conv}
\big| \nabla \lambda_i \cdot r_i(u(t_0,x)) \big| \geq \frac{c_0}{2} > 0,
\end{equation}
by taking $b'_0 \ll 1$ in order to have that the total variation remains sufficiently small.

Since $u \llcorner_{T_0}$ coincides with the solution to
\begin{equation}
\label{e:local}
\left\{
\begin{array}{l}
\partial_t w + f(w)_x = 0,  \\
w(x,t_0)=\begin{cases}
u_{t_0}(x) & |x-x_0| < b'_0, \\
\frac{1}{2b_0'} \int^{x_0+b_0'}_{x_0-b_0'} u_{t_0}(y)dy & |x-x_0| \geq b_0,
\end{cases}
\end{array} \right.
\end{equation}
and by taking $b_0'$ sufficiently small, we still have that \eqref{e:uni_conv} holds for the range of $w$. In particular $w$ is SBV outside a countable number of times, and the same happens for $u$ in $T_0$.

%

As in the scalar case, one thus verifies that there is a countable family of triangles $\{T_i\}_{i=1}^\infty$ covering the complement of $C$ outside a set whose projection on the $t$-axis is countable. The same computation of the scalar case concludes the proof: for any $\tau$ chosen as in \eqref{e:final_SBV}
%
\begin{eqnarray*}
\big| D^\mathrm{c}\lambda_i \cdot r_i \big|(\R) \leq \big| D^\mathrm{c} \lambda_i \cdot r_i \big| (C_\tau) + \big|D^\mathrm{c} \lambda_i \cdot r_i \big| \bigg( \bigcup_i T_i\cap\{t=\tau\} \bigg) = 0.
\end{eqnarray*}
\end{proof}

Similar to the scalar case, it is easy to get the following corollary from the Theorem \ref{t2}.

\begin{corollary}
Let $u$ be the vanishing viscosity solution of the problem \eqref{e:basic}, \eqref{e:initial}. Then the measure $(D_u \lambda_i(u) \cdot r_i(u)) (l_i(u) \cdot u_x)$ has no Cantor part in $\R^+ \times \R$.
\end{corollary}

\begin{remark}
\label{ex:nonsbv}
Consider the following equations
\begin{equation*}
\begin{cases}
u_t=0, \\
v_t+((1 + v + u) v )_x=0.
 \end{cases}
\end{equation*}
Since $D_x \lambda_2((u,v))= D_x u+ 2 D_x v$, then it is clear that $D_x \lambda_2$ can have a Cantor part, while from Theorem \ref{t:me} the second component
\[
\big( D_u \lambda_2 \cdot r_2 \big) \big( l_2 \cdot (u,v)_x \big) = \frac{2}{1+u+2v} \big( v u_x + (1+u+2v) v_x \big)
\]
has not Cantor part.
\end{remark}

\bibliographystyle{plain}

\end{document}